\documentclass[12pt,reqno]{amsart}
\usepackage[left=80pt,right=80pt]{geometry}
\usepackage[usenames]{color}
\usepackage{amsmath}
\usepackage{amssymb}
\usepackage{amsthm}
\usepackage{enumitem}
\usepackage{float}
\usepackage{graphicx}
\usepackage{tikz}
\usepackage{cases}
\usepackage{tikz-qtree}
\usetikzlibrary{graphs,graphs.standard,calc}
\usepackage[caption=false]{subfig}
\usepackage{hyperref,url}
\hypersetup{
	colorlinks=true,
  linkcolor=black,          
  citecolor=black,         
  filecolor=black,      
  urlcolor=black           
}

\newtheorem{prop}{Proposition}
\newtheorem{lemma}[prop]{Lemma}
\newtheorem{theorem}[prop]{Theorem}

\newtheorem{corollary}[prop]{Corollary}
\theoremstyle{definition}

\newtheorem{remark}[prop]{Remark}
\newtheorem{example}[prop]{Example}

\newcommand{\M}{\mathcal{M}}
\newcommand{\LL}{\mathcal{L}}
\newcommand{\seqnum}[1]{\href{https://oeis.org/#1}{\rm \underline{#1}}}
\allowdisplaybreaks
\newcommand{\mylabel}[2]{#2\def\@currentlabel{#2}\label{#1}}

\begin{document}
\tikzset{mystyle/.style={matrix of nodes,
        nodes in empty cells,
        row 1/.style={nodes={draw=none}},
        row sep=-\pgflinewidth,
        column sep=-\pgflinewidth,
        nodes={draw,minimum width=1cm,minimum height=1cm,anchor=center}}}
\tikzset{mystyleb/.style={matrix of nodes,
        nodes in empty cells,
        row sep=-\pgflinewidth,
        column sep=-\pgflinewidth,
        nodes={draw,minimum width=1cm,minimum height=1cm,anchor=center}}}

\title{Random Walk Labelings of Perfect Trees and Other Graphs}

\author[SELA FRIED]{Sela Fried$^{\dagger}$}
\thanks{$^{\dagger}$ Department of Computer Science, Israel Academic College,
52275 Ramat Gan, Israel.
\\
\href{mailto:friedsela@gmail.com}{\tt friedsela@gmail.com}}
\author[TOUFIK MANSOUR]{Toufik Mansour$^{\sharp}$}
\thanks{$^{\sharp}$ Department of Mathematics, University of Haifa, 3498838 Haifa,
Israel.\\
\href{mailto:tmansour@univ.haifa.ac.il}{\tt tmansour@univ.haifa.ac.il}}

\maketitle

\begin{abstract}
A Random walk labeling of a graph $G$ is any labeling of $G$ that could have been obtained by performing a random walk on $G$. Continuing two recent works, we calculate the number of random walk labelings of perfect trees, combs, and double combs, the torus $C_2\times C_n$, and the graph obtained by connecting three path graphs to form two cycles.
\bigskip


\smallskip

\noindent \textbf{Math.~Subj.~Class.:} 05C78, 05A10, 05A15, 05C81.
\end{abstract}

\section{Introduction}
In \cite{Fr1} and \cite{Fr2} we have defined and studied the notion of random walk labelings. These are graph labelings obtainable by performing random walks on graphs. More precisely, suppose $G$ is a connected and undirected graph with vertex set $V$. A \emph{random walk labeling} of $G$ is a labeling that is obtainable by performing the following process:
\begin{enumerate}
    \item Set $i=1$ and let $v\in V$. Label $v$ with $i$.
    \item As long as there are vertices that are not labeled, pick $w\in V$ that is adjacent to $v$ and replace $v$ with $w$. If $v$ is not labeled, increase $i$ by $1$ and label $v$ with $i$.
\end{enumerate}
The number of random walk labelings of a graph $G$ is denoted by $\LL(G)$. In the two works mentioned above, we calculated the number of random walk labelings of many graph families. In this work, we calculate the number of random walk labelings of perfect trees, combs, and double combs, the torus $C_2\times C_n$, and the graph obtained by connecting three path graphs to form two cycles.
\section{Main results}

\subsection{Perfect trees}
For integers $m\geq 2$ and $h\geq 0$ let us denote by $T_{h,m}$ the perfect $m$-ary tree of height $h$.

\begin{theorem}
We have $$\mathcal{L}(T_{h,m})=\sum_{k=0}^{h}m^{k}\gamma_{h,m,k}\left(\prod_{j=1}^{h-k}\alpha_{j,m}^{m^{h-k-j}}\right)\prod_{j=0}^{k-1}\beta_{h,m,j}\prod_{\ell=1}^{h-1-j}\alpha_{\ell,m}^{(m-1)m^{h-1-j-\ell}},$$ where
\begin{align*}
\alpha_{h,m}&=\dbinom{\frac{m^{h+1}-m}{m-1}}{\underbrace{\textstyle\frac{m^{h}-1}{m-1},\ldots,\frac{m^{h}-1}{m-1}}_{m \textnormal{ times}}},\\
\beta_{h,m,k}&=\dbinom{\frac{m^{h+1}-m^{h-k}}{m-1}-1}{\underbrace{\textstyle\frac{m^{h-k}-1}{m-1},\ldots,\frac{m^{h-k}-1}{m-1}}_{m-1 \textnormal{ times}}, \frac{m^{h+1}-m^{h-k+1}}{m-1}},\\
\gamma_{h,m,k}&=\dbinom{\frac{m^{h+1}-m}{m-1}}{\frac{m^{h+1-k}-m}{m-1}}.
\end{align*}
\end{theorem}

\begin{proof}
For $0\leq k\leq h$ we denote by $t_{h,m,k}$ the number of random walk labelings of $T_{h,m}$ that begin at a vertex of depth $k$. Clearly, \begin{equation}\label{eq;gkb}
\mathcal{L}(T_{h,m})=\sum_{k=0}^{h}m^{k}t_{h,m,k}.
\end{equation}
Assume that $h\geq 1$ and $0\leq k\leq h-1$. We denote by $S_{h,m,k}$ the graph obtained from $T_{h,m,k}$ by deleting one child, together with its descendants, of a vertex of depth $k$ (see figure \ref{fig:310} below). We denote by $s_{h,m,k}$ the number of random walk labelings of $S_{h,m,k}$ that begin at the vertex whose child has been deleted.
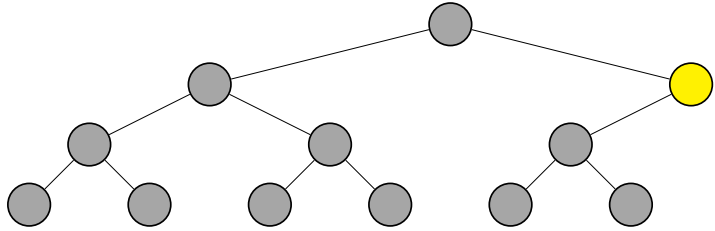
\begin{figure}[H]
\centering
\scalebox{0.8}{
\begin{tikzpicture}[shape = circle, node distance=4cm and 7cm,  nodes={ circle,fill=black!35}]
\node[draw,circle,inner sep=0.25cm] (0) at (0, 0) [thick] {};
\node[draw,circle,inner sep=0.25cm] (1) at (-4, -1) [thick] {};
\node[draw, circle,inner sep=0.25cm] (2) at (-6, -2) [thick] {};
\node[draw,circle,inner sep=0.25cm] (3) at (-2, -2) [thick] {};
\node[draw,circle,inner sep=0.25cm] (4) at (-7, -3) [thick] {};
\node[draw,circle,inner sep=0.25cm] (5) at (-5, -3) [thick] {};
\node[draw,circle,inner sep=0.25cm] (6) at (-3, -3) [thick] {};
\node[draw,circle,inner sep=0.25cm] (7) at (-1, -3) [thick] {};
\node[draw, fill=yellow,circle,inner sep=0.25cm] (8) at (4, -1) [thick] {};
\node[draw,circle,inner sep=0.25cm] (9) at (2, -2) [thick] {};
\node[draw,circle,inner sep=0.25cm] (11) at (1, -3) [thick] {};
\node[draw,circle,inner sep=0.25cm] (12) at (3, -3) [thick] {};
\path (0) edge[ultra thin,-] (1);
\path (1) edge[ultra thin,-] (2);
\path (1) edge[ultra thin,-] (3);
\path (2) edge[ultra thin,-] (4);
\path (2) edge[ultra thin,-] (5);
\path (3) edge[ultra thin,-] (6);
\path (3) edge[ultra thin,-] (7);
\path (0) edge[ultra thin,-] (8);
\path (8) edge[ultra thin,-] (9);
\path (9) edge[ultra thin,-] (11);
\path (9) edge[ultra thin,-] (12);
\end{tikzpicture}}
\caption{The number of random walk labelings of $S_{3,2,1}$ that begin at the yellow vertex is $s_{3,2,1}$.}\label{fig:310}
\end{figure} \noindent We obtain two sequences $\{s_{h,m,k}\}_{\substack{h\geq 1\\0\leq k\leq h-1}}$ and $\{t_{h,m,k}\}_{\substack{h\geq 0\\0\leq k\leq h}}$ that satisfy the following relations for, $h\geq 1$:

\begin{numcases}{s_{h,m,k}=}
   \left(t_{h-1,m,0}\right)^{m-1}\dbinom{m^{h}-1}{\underbrace{\textstyle\frac{m^{h}-1}{m-1},\ldots,\frac{m^{h}-1}{m-1}}_{m-1 \textnormal{ times}} } & if $k=0$ \label{R1}\\
   \left(t_{h-k-1,m,0}\right)^{m-1}s_{h,m,k-1}\dbinom{\frac{m^{h+1}-m^{h-k}}{m-1}-1}{{\underbrace{\textstyle \frac{m^{h-k}-1}{m-1},\ldots,\frac{m^{h-k}-1}{m-1}}_{m-1 \textnormal{ times}}}, \frac{m^{h+1}-m^{h-k+1}}{m-1}} & if $1\leq k\leq h-1$ \label{R2}
\end{numcases}
\begin{numcases}{t_{h,m,k}=}
   \left(t_{h-1,m,0}\right)^m\dbinom{\frac{m^{h+1}-1}{m-1}-1}{
   \underbrace{\textstyle\frac{m^{h}-1}{m-1},\ldots,\frac{m^{h}-1}{m-1}}_{m \textnormal{ times}}
   } & if $k=0$ \label{R3}\\
   t_{h-k,m,0}s_{h,m,k-1}\dbinom{\frac{m^{h+1}-1}{m-1}-1}{\frac{m^{h-k+1}-1}{m-1}-1}& if $1\leq k\leq h$ \label{R4}
\end{numcases} The initial condition is \begin{equation}\label{eq;0010}
t_{0,m,0}=1.\end{equation}
It is not hard to see why these relations hold. Let us, as an example, prove (\ref{R2}). First, recall that in an $m$-ary perfect tree of height $h$ there are $(m^{h+1}-1)/(m-1)$ vertices. By definition of $s_{h,m,k}$, the random walk labeling begins at a vertex $x$ of depth $k$ with only $m-1$ children. Each of these children is the root of a tree $T_{h-k-1,m}$, whose random walk labeling begins at the root. This gives the term $\left(t_{h-k-1,m,0}\right)^{m-1}$. The graph obtained from $S_{h,m,k}$ by deleting $x$, together with its $m-1$ children, is exactly $S_{h,m,k-1}$. This gives the term $s_{h,m,k-1}$. Now, the number of vertices of $S_{h,m,k-1}$ is
$$\frac{m^{h+1}-1}{m-1}-\frac{m^{h-k+1}-1}{m-1}= \frac{m^{h+1}-m^{h-k+1}}{m-1}$$ and each of the $m-1$ trees $T_{h-k-1,m}$ has $(m^{h-k}-1)/(m-1)$ vertices. This gives the multinomial term and concludes the proof of this relation.

We now solve the above recursions to obtain a closed formula for $\mathcal{L}(T_{h,m})$. Inductively, it follows from (\ref{R3}) and (\ref{eq;0010}) that \begin{equation}\label{R20}t_{h,m,0}=\prod_{j=1}^h\alpha_{j,m}^{m^{h-j}}\end{equation} and, from (\ref{R1}), that
\begin{equation}\label{R10}
s_{h,m,0}=\beta_{h,m,0}\prod_{j=1}^{h-1}\alpha_{j,m}^{(m-1)m^{h-1-j}}.\end{equation} Therefore, by (\ref{R2}) and (\ref{R10}),
\begin{equation}\label{R11}s_{h,m,k}=\prod_{j=0}^{k}\beta_{h,m,j}\prod_{\ell=1}^{h-1-j}\alpha_{\ell,m}^{(m-1)m^{h-1-j-\ell}},\end{equation} which holds for every $h\geq 1$ and $0\leq k\leq h-1$. Finally, by (\ref{R4}), (\ref{R20}) and (\ref{R11}),
$$t_{h,m,k}=
\gamma_{h,m,k}\left(\prod_{j=1}^{h-k}\alpha_{j,m}^{m^{h-k-j}}\right)\prod_{j=0}^{k-1}\beta_{h,m,j}\prod_{\ell=1}^{h-1-j}\alpha_{\ell,m}^{(m-1)m^{h-1-j-\ell}},$$ which holds for every $h\geq 0$ and $0\leq k\leq h$. The assertion follows now from (\ref{eq;gkb}).
\end{proof}

\begin{remark}
The sequence $\{t_{h,2,0}\}_{h\geq 0}$  coincides with sequence $\seqnum{A056972}(h+1)$ in the On-Line Encyclopedia of Integer Sequences (OEIS) \cite{OL}.
\end{remark}

\begin{example}
The first six values of $\mathcal{L}(T_{h,2})$ are
\begin{align*}
&1,\\ &4,\\ &240,\\ &82368000,\\ &315717859104620544000000,\\ &11684127387646867268494413939618462518646164707029811200000000000.
\end{align*}
\end{example}

\subsection{Combs and double combs}

For integers $m,n\geq 2$ and $1\leq k\leq n$ let us denote by $C_{m,n,k}$ the graph obtained by connecting $m$ path graphs $P_n$ through their $k$th vertex (see figure \ref{fig:300} below for a visualization). This construction generalizes two well-known graph families called combs and double combs (e.g., \cite[p.~334]{G}).
We shall refer to vertices of $C_{m,n,k}$ as tuples $(i,j)$, where $1\leq i\leq m$ and $1\leq j\leq n$.
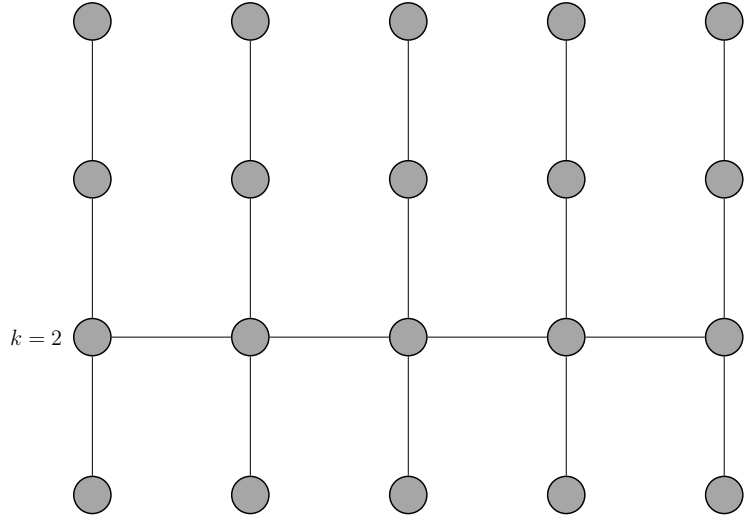
\begin{figure}[H]
\centering
\scalebox{0.7}{
\begin{tikzpicture}[shape = circle, node distance=4cm and 7cm,  nodes={ circle,fill=black!35}]
  \foreach \y in {1,...,4}{%
    \foreach \x in {1,...,5}{%
    \node[draw,circle,inner sep=0.25cm] (N-\x-\y) at (3*\x, 3*\y) [thick] {};
  }
  }
\foreach \from/\to in {1/2,2/3,3/4}{%
    \foreach \x in {1,...,5}{%
    \path (N-\x-\from) edge[ultra thin,-] (N-\x-\to);
  }
  }
\foreach \from/\to in {1/2,2/3,3/4,4/5}{%
    \path (N-\from-2) edge[ultra thin,-] (N-\to-2);
  }
  \node also [label={[fill=white]left:$k=2$}] (N-1-2);
\end{tikzpicture}}
\caption{The graph $C_{5,4,2}$.}\label{fig:300}
\end{figure} \noindent

In the calculation of $\mathcal{L}(C_{m,n,k})$ we shall make use of the following combinatorial identity.

\begin{lemma}\label{lem;pac}
Let $n\geq 1$ be an integer. Consider the sequence $\{t_{m,n,k}\}_{\substack{m\geq 0\\ 1\leq k\leq n}}$ defined by $$t_{m,n,k} =  \binom{n-1}{k-1}^{m}\prod_{\ell=1}^{m-1}\binom{(\ell+1)n-1}{n-1}.$$ Then $$\sum_{j=0}^{m}\frac{t_{j,n,k}}{(jn)!}\frac{t_{m-j,n,k}}{((m-j)n)!}=\frac{1}{m!}\left(\frac{2}{n!}\binom{n-1}{k-1}\right)^m.$$
\end{lemma}

\begin{proof}
Define $T_{n,k}(x)=\sum_{m\geq0}t_{m,n,k}\frac{x^m}{(mn)!}$. Clearly, $t_{m+1,n,k}=\binom{n-1}{k-1}\binom{(m+1)n-1}{n-1}t_{m,n,k}$. Multiplying both sides of this equation by $\frac{x^{m+1}}{((m+1)n)!}$ and summing over $m\geq 0$, we obtain the integral equation $T_{n,k}(x)=1+\frac{1}{n!}\binom{n-1}{k-1}\int_0^xT_{n,k}(u)du$, whose unique solution is given by
$$T_{n,k}(x)=e^{\frac{1}{n!}\binom{n-1}{k-1}x}.$$
It follows that $T^2_{n,k}(x)=e^{\frac{2}{n!}\binom{n-1}{k-1}x}$, which proves the assertion.
\end{proof}

\begin{theorem}\label{thm;ab9}
We have
\begin{align*}
&\mathcal{L}(C_{m,n,k})=\frac{1}{(m-1)!}\left(\frac{2}{n!}\binom{n-1}{k-1}\right)^{m-1}\Bigg(\frac{1}{(n-k)!}\sum_{y=2}^{k}\frac{2^{y-2}(mn-y)!}{(k-y)!}+\\
&\frac{1}{(k-1)!}\sum_{y=2}^{n-k+1}\frac{2^{y-2}(mn-y)!}{(n-k+1-y)!}+\frac{(mn-1)!}{(n-k)!(k-1)!}\Bigg)	
\end{align*}
\end{theorem}

\begin{proof}
Let us denote by $t_{m,n,k}$ the number of random walk labelings of $C_{m,n,k}$ that begin at vertex $(1,k)$. We claim that $$t_{m,n,k} = \binom{n-1}{k-1}^{m}\prod_{\ell=1}^{m-1}\binom{(\ell+1)n-1}{n-1}.$$ Indeed, if $m=1$ then, by \cite[Example 1(b)]{Fr1}, $t_{1,n,k} = \binom{n-1}{k-1}$ and equality holds in this case. Suppose that the claim holds for $m$. There are $\binom{n-1}{k-1}$ random walk labelings of the first path graph and, by the induction hypothesis, $t_{m,n,k}$ random walk labelings of the remaining $m$ path graphs. Finally, there are $\binom{(m+1)n-1}{n-1}$ ways to alternate between the two random walk labelings. It follows that
\begin{align*}
t_{m+1,n,k} &= \binom{n-1}{k-1}\binom{(m+1)n-1}{n-1}t_{m,n,k}\\&=\binom{n-1}{k-1}^{m+1}\prod_{\ell=1}^{m}\binom{(\ell+1)n-1}{n-1}
\end{align*} and the assertion follows. It will be convenient to set $t_{0,n,k}=1$.

Assume now that the random walk labeling begins at vertex $(j,s)$, where $1\leq j\leq m$ and $1\leq s< k$. Assume also that the label of vertex $(j,k)$ is $y$, where $k-s+1\leq y\leq k$ (see figure \ref{fig:30p} below).
We claim that the number of such random walk labelings is given by
\begin{equation}\label{eq;90a}
\binom{y-2}{k-s-1}\binom{mn-y}{(j-1)n,n-y,(m-j)n}\binom{n-y}{n-k}t_{j-1,n,k}t_{m-j,n,k}.
\end{equation}

\begin{figure}[H]
\centering
\scalebox{0.7}{
\begin{tikzpicture}[shape = circle, node distance=4cm and 7cm,  nodes={ circle,fill=black!35}]
  \foreach \y in {1,...,4}{%
    \foreach \x in {1,...,5}{%
    \node[draw,circle,inner sep=0.25cm] (N-\x-\y) at (3*\x, 3*\y) [thick] {};
  }
  }
\foreach \from/\to in {1/2,2/3,3/4}{%
    \foreach \x in {1,...,5}{%
    \path (N-\x-\from) edge[ultra thin, -]  node[ fill=white, anchor=center, pos=0.5] {$\vdots$} (N-\x-\to) ;

  }
  }
\foreach \from/\to in {1/2,2/3,3/4,4/5}{%
    \path (N-\from-3) edge[ultra thin, -] node[ fill=white, anchor=center, pos=0.5] {$\cdots$} (N-\to-3);
  }
\node also [label={[fill=white]left:$k$}] (N-1-3);
\node also [label={[fill=white]left:$s$}] (N-1-2);
\node also [label={[fill=white]below:$j$}] (N-3-1);
\node[draw,circle,inner sep=0.25cm,label=center:{$1$}] (N-3-2) at (9, 6) [thick] {};
\node[draw,circle,inner sep=0.25cm,label=center:{$y$}] (N-3-3) at (9, 9) [thick] {};
\node also [label={[fill=white]left:$1$}] (N-1-1);
\node also [label={[fill=white]left:$n$}] (N-1-4);
\end{tikzpicture}}
\caption{The graph $C_{m,n,k}$.}\label{fig:30p}
\end{figure}
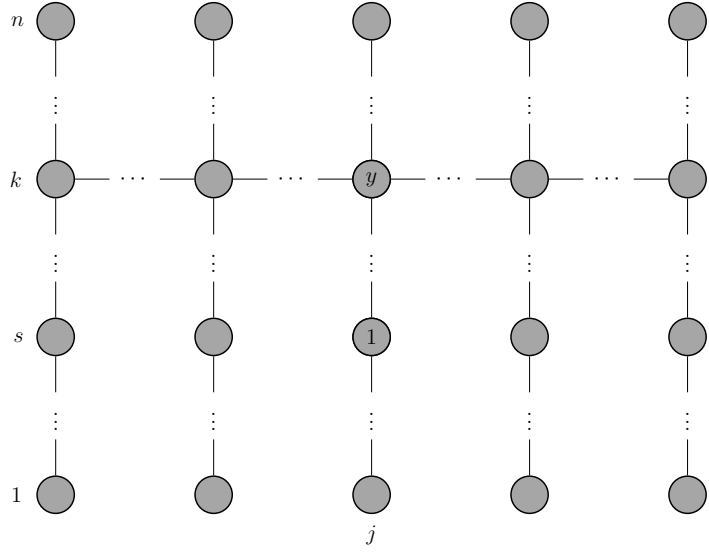 \noindent Indeed, there are $\binom{y-2}{k-s-1}$ ways to reach vertex $(j,k)$ from vertex $(j,s)$, such that the label of the former is $y$. Once the vertex $(j,k)$ is reached, the graph naturally splits into three subgraphs: one isomorphic to $C_{j-1,n,k}$, one isomorphic to $C_{m-j,n,k}$ and one isomorphic to $P_n$, which is partially labeled. The random walk labeling of $C_{j-1,n,k}$ must begin at vertex $(j-1,k)$. By the previous part, there are $t_{j-1,n,k}$ such random walk labelings. Similarly, there are $t_{n-j,n,k}$ random walk labelings of $C_{m-j,n,k}$ that begin at vertex $(j+1,k)$. Clearly, there are $\binom{n-y}{n-k}$ ways to label the remaining $n-y$ vertices of $P_n$. Finally, there are $\binom{mn-y}{(j-1)n,n-y,(m-j)n}$ ways to alternate between the three random walk labelings.

Let us write $$A_{j,k,y}:=\binom{mn-y}{(j-1)n,n-y,(m-j)n}\binom{n-y}{n-k}t_{j-1,n,k}t_{m-j,n,k}.$$
It follows that the number of random walk labelings of $C_{m,n,k}$, that start at vertex $(j, s)$, is given by
\begin{equation}\label{eq;9341}
    \sum_{y=k-s+1}^{k}\binom{y-2}{k-s-1}A_{j,k,y}.
\end{equation}
Exploiting the fact that $C_{m,n,k}$ is isomorphic to $C_{m,n,n-k+1}$ and using similar reasoning for random walk labelings that begin at a vertex $(j, k)$, for $1\leq j\leq m$, we conclude that
\begin{align}
&\mathcal{L}(C_{m,n,k})\nonumber\\
&=
\sum_{j=1}^{m}\left(A_{j,k,1}+\sum_{s=1}^{k-1}\sum_{y=k-s+1}^{k}\binom{y-2}{k-s-1}A_{j,k,y}+	
\sum_{s=1}^{n-k}\sum_{y=n-k+2-s}^{n-k+1}\binom{y-2}{n-k-s}A_{j,n-k+1,y}\right)\nonumber\\
&=\sum_{j=1}^{m}\left(A_{j,k,1}+\sum_{y=2}^k\sum_{s=k+1-y}^{k-1}\binom{y-2}{k-s-1}A_{j,k,y}+\sum_{y=2}^{n-k+1}\sum_{s=n-k+2-y}^{n-k}\binom{y-2}{n-k-s}A_{j,n-k+1,y}\right)\nonumber\\
&=\sum_{j=1}^{m}\left(A_{j,k,1}+\sum_{y=2}^k2^{y-2}A_{j,k,y}	
+\sum_{y=2}^{n-k+1}2^{y-2}A_{j,n-k+1,y}\right)\nonumber\\
&=\sum_{j=1}^{m}A_{j,k,1}+\sum_{y=2}^k2^{y-2}\sum_{j=1}^{m}A_{j,k,y}	
+\sum_{y=2}^{n-k+1}2^{y-2}\sum_{j=1}^{m}A_{j,n-k+1,y}\label{eq;e1}
\end{align} Now, using Lemma \ref{lem;pac},
\begin{align*}
\sum_{j=1}^{m}A_{j,k,1} &= \frac{(mn-1)!}{(n-k)!(k-1)!}\sum_{j=0}^{m-1}\frac{t_{j,n,k}}{(jn)!}\frac{t_{m-1-j,n,k}}{((m-1-j)n)!}\\&=\frac{(mn-y)!}{(n-k)!(k-y)!(m-1)!}\left(\frac{2}{n!}\binom{n-1}{k-1}\right)^{m-1}.
\end{align*}
Similarly,
\begin{align*}
\sum_{j=1}^{m}A_{j,k,y}&=\frac{(mn-y)!}{(n-k)!(k-y)!(m-1)!}\left(\frac{2}{n!}\binom{n-1}{k-1}\right)^{m-1},\\
\sum_{j=1}^{m}A_{j,n-k+1,y} &=\frac{(mn-y)!}{(k-1)!(n-k+1-y)!(m-1)!}\left(\frac{2}{n!}\binom{n-1}{n-k}\right)^{m-1}.
\end{align*}
Substituting these into (\ref{eq;e1}) finishes the proof.
\end{proof}

Two special cases of Theorem \ref{thm;ab9} are of particular interest, namely, $n=2,k=1$ and $n=3,k=2$, giving combs and double combs, respectively. The former seems to give the first combinatorial interpretation to sequence \seqnum{A151817} in the OEIS.

\begin{corollary}
We have
    $$\mathcal{L}(C_{m,2,1})=2^{m-1}m(m-1)!!$$
and
    $$\mathcal{L}(C_{m,3,2})=\frac{2^{m-1}(3m+1)!}{3^{m}(3m-1)m!}.$$
\end{corollary}

\subsection{The graph \texorpdfstring{$C_2\times C_n$}{}}

Let $n\geq 1$ be an integer. In \cite[Theorem 3.2]{Fr1} we calculated the number of random walks of the grid graph $P_2\times P_n\simeq C_2\times P_n$. Here, we study its toroidal version.

\begin{theorem}
For $n\geq2$ we have
$$\mathcal{L}(C_2\times C_n)=\frac{n(n+2)(2n-2)!}{(n-2)!}.$$
\end{theorem}
\begin{proof}
For $1\leq k\leq n$ let us denote by $a_{n,k}$ the number of ways to resume a random walk labeling, provided that $k$ vertices in one row have already been labeled. Clearly, $\mathcal{L}(C_2\times C_n)=2na_{n,1}$.
Now, for $0\leq s,t\leq n-1$, denote by $b_{n,s,t}$ the number of ways to resume a random walk labeling, provided that $s+1$ vertices in one row and $t+1$ vertices in the other row have already been labeled, such that exactly two vertices in different rows are adjacent (see figure \ref{fig;101} below).
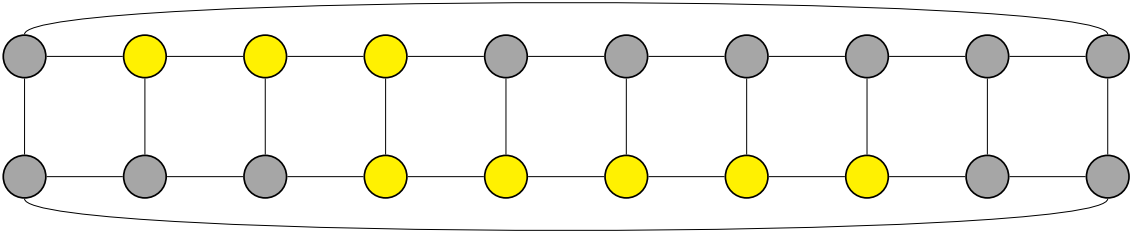
\begin{figure}[H]
\centering
\scalebox{0.8}{
\begin{tikzpicture}[shape = circle, node distance=4cm and 7cm,  nodes={ circle,fill=black!35}]

\node[draw,circle,inner sep=0.25cm] (0) at (0, -2) [thick] {};
\node[draw,circle,inner sep=0.25cm] (1) at (2, -2) [thick] {};
\node[draw,circle,inner sep=0.25cm] (2) at (4, -2) [thick] {};
\node[draw, fill=yellow,circle,inner sep=0.25cm] (3) at (6, -2) [thick] {};
\node[draw, fill=yellow,circle,inner sep=0.25cm] (4) at (8, -2) [thick] {};
\node[draw, fill=yellow,circle,inner sep=0.25cm] (5) at (10, -2) [thick] {};
\node[draw, fill=yellow,circle,inner sep=0.25cm] (6) at (12, -2) [thick] {};
\node[draw, fill=yellow,circle,inner sep=0.25cm] (7) at (14, -2) [thick] {};
\node[draw,circle,inner sep=0.25cm] (8) at (16, -2) [thick] {};
\node[draw,circle,inner sep=0.25cm] (9) at (18, -2) [thick] {};

\node[draw,circle,inner sep=0.25cm] (00) at (0, 0) [thick] {};
\node[draw, fill=yellow,circle,inner sep=0.25cm] (10) at (2, 0) [thick] {};
\node[draw, fill=yellow,circle,inner sep=0.25cm] (20) at (4, 0) [thick] {};
\node[draw, fill=yellow,circle,inner sep=0.25cm] (30) at (6, 0) [thick] {};
\node[draw,circle,inner sep=0.25cm] (40) at (8, 0) [thick] {};
\node[draw,circle,inner sep=0.25cm] (50) at (10, 0) [thick] {};
\node[draw,circle,inner sep=0.25cm] (60) at (12, 0) [thick] {};
\node[draw,circle,inner sep=0.25cm] (70) at (14, 0) [thick] {};
\node[draw,circle,inner sep=0.25cm] (80) at (16, 0) [thick] {};
\node[draw,circle,inner sep=0.25cm] (90) at (18, 0) [thick] {};
\path (0) edge[ultra thin,-] (1);
\path (1) edge[ultra thin,-] (2);
\path (2) edge[ultra thin,-] (3);
\path (3) edge[ultra thin,-] (4);
\path (4) edge[ultra thin,-] (5);
\path (5) edge[ultra thin,-] (6);
\path (6) edge[ultra thin,-] (7);
\path (7) edge[ultra thin,-] (8);
\path (8) edge[ultra thin,-] (9);
\path (00) edge[ultra thin,-] (10);
\path (10) edge[ultra thin,-] (20);
\path (20) edge[ultra thin,-] (30);
\path (30) edge[ultra thin,-] (40);
\path (40) edge[ultra thin,-] (50);
\path (50) edge[ultra thin,-] (60);
\path (60) edge[ultra thin,-] (70);
\path (70) edge[ultra thin,-] (80);
\path (80) edge[ultra thin,-] (90);
\path (00) edge[ultra thin,-] (0);
\path (10) edge[ultra thin,-] (1);
\path (20) edge[ultra thin,-] (2);
\path (30) edge[ultra thin,-] (3);
\path (40) edge[ultra thin,-] (4);
\path (50) edge[ultra thin,-] (5);
\path (60) edge[ultra thin,-] (6);
\path (70) edge[ultra thin,-] (7);
\path (80) edge[ultra thin,-] (8);
\path (90) edge[ultra thin,-] (9);
\draw[-] (00) to[out=90,in=90,  looseness=0.1] (90);
\draw[-] (0) to[out=-90,in=-90,  looseness=0.1] (9);
\end{tikzpicture}}
\caption{The number of ways to resume the random walk labeling, in this case, is given by $b_{10,2,4}$.}\label{fig;101}
\end{figure}

We obtain two sequences $\{a_{n,k}\}_{\substack{n\geq 1\\1\leq k\leq n}}$ and $\{b_{n,s,t}\}_{\substack{n\geq 1\\0\leq s,t\leq n-1}}$ that satisfy the following relations for, $n\geq 3$:
{\small\begin{align}\label{reca1}
    a_{n,k}&=\begin{cases}
        2a_{n,k+1} + b_{n,0,0} & \textnormal{if } k = 1\\
        2a_{n,k+1} + (k-2)a_{n-1,k-1}+2b_{n,k-1,0}& \textnormal{if } 2\leq k\leq n-2\\ a_{n,k+1} + (n-3)a_{n-1,k-1}+2b_{n,k-1,0}& \textnormal{if } k=n-1\\ n!& k=n
    \end{cases}
\end{align}}
and
{\small\begin{align}\label{recb1}
b_{n,s,t}&=
\begin{cases}
4b_{n,1,0}&\textnormal{if } s=0 \textnormal{ and } t=0 \\
&\\
(n-1)!&\textnormal{if } s=0 \textnormal{ and } t= n -1 \\
b_{n,1,n-2}+b_{n,0,n-1}+\sum_{q=1}^{n-2}\binom{n-1}{q-1}(q-1)!b_{n-q,0,n-2-q}&\textnormal{if } s=0 \textnormal{ and } t= n - 2 \\
a_{n-1,t+1}+b_{n,1,t}+b_{n,0,t+1}+ \sum_{q=1}^{t}\binom{2n-(t+3)}{q-1}(q-1)!b_{n-q,0,t-q}&\textnormal{if } s=0 \textnormal{ and } 1\leq t \leq  n - 3 \\
&\\
(n-1)!&\textnormal{if } s=n-1 \textnormal{ and } t= 0 \\
b_{n,n-2,1}+b_{n,n-1,0}+\sum_{q=1}^{n-2}\binom{n-1}{q-1}(q-1)!b_{n-q,n-2-q,0}&\textnormal{if } s=n-2 \textnormal{ and } t= 0 \\
a_{n-1,s+1}+b_{n,s,1}+b_{n,s+1,0}+\sum_{q=1}^{s}\binom{2n-(s+3)}{q-1}(q-1)!b_{n-q,s-q,0}&\textnormal{if } 1\leq s \leq  n - 3 \textnormal{ and } t=0 \\
&\\
b_{n,s+1,t}+\sum_{q=1}^{s}\binom{2n-(s+t+3)}{q-1}(q-1)!b_{n-q,s-q,t}+&\\b_{n,s,t+1}+\sum_{q=1}^{t}\binom{2n-(s+t+3)}{q-1}(q-1)!b_{n-q,s,t-q}&\textnormal{if } s,t>0 \textnormal{ and } s+t\leq n-2 \\
(n-1)! & \textnormal{if } s+t= n-1 \\
0 & \textnormal{if } s+t\geq n
\end{cases}
\end{align}}
The initial conditions are
\begin{align*}
    a_{1,1} & = 1,\\
    a_{2,1} & = 4, a_{2,2} = 2, \\
    b_{1,0,0} & =0,\\
    b_{2,0,0} & =2,b_{2,1,0} = b_{2,0,1} = 1, b_{2,1,1} = 0.\\
\end{align*}
As in the case of perfect trees, we prove only one of the relations, say, the $7$th case of $b_{n,s,t}$. In this case $s+1$ vertices are labeled in the first row and one vertex is labeled in the second row (see Figure \ref{fig;101ab} below for a visualization). There are $s+3$ possible vertices for the next label, that we denote by $p_1,p_2,p_3$ and $r_1,\ldots,r_s$. If we proceed with $p_2$ or $p_3$, the number of ways to complete the random walk labeling is given by $b_{n,s+1,0}$ and $b_{n,s,1}$, respectively. If we proceed with $p_1$, we may delete the column left of $p_1$ and obtain a graph isomorphic to $C_2\times C_{n-1}$ in which $s+1$ vertices in the first row are already labeled. The number of ways to continue from this setting is given by $a_{n-1,s+1}$. Finally, let $1\leq q\leq s$ and assume that we proceed with $r_q$. There are $(q-1)!$ ways to label the vertices $r_1,\ldots,r_{q-1}$. Deleting the columns corresponding to the vertices $r_1,\ldots,r_q$, we obtain a graph isomorphic to $C_2\times C_{n-q}$ in which $s-q+1$ vertices are labeled in the first row and one vertex is labeled in the second row. The number of ways to continue from this setting is given by $b_{n-q,s-q,0}$. Now, once $r_q$ is labeled, there remain $2n-(s+3)$ unlabeled vertices and we need to decide when to label the vertices $r_1,\ldots,r_{q-1}$. This gives the binomial term in this relation and concludes the proof of this case.
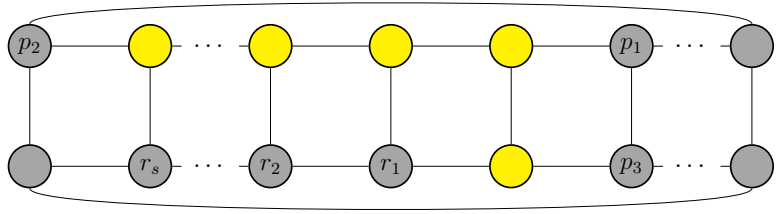
\begin{figure}[H]
\centering
\scalebox{0.8}{
\begin{tikzpicture}[shape = circle, node distance=4cm and 7cm,  nodes={ circle,fill=black!35}]

\node[draw,circle,inner sep=0.25cm,label=center:{$p_2$}] (2) at (4, 2) [thick] {};

\node[draw, fill=yellow,circle,inner sep=0.25cm] (3) at (6, 2) [thick] {};
\node[draw, fill=yellow,circle,inner sep=0.25cm] (4) at (8, 2) [thick] {};
\node[draw, fill=yellow,circle,inner sep=0.25cm] (5) at (10, 2) [thick] {};
\node[draw, fill=yellow,circle,inner sep=0.25cm] (6) at (12, 2) [thick] {};
\node[draw,circle,inner sep=0.25cm,label=center:{$p_1$}] (7) at (14, 2) [thick] {};
\node[draw,circle,inner sep=0.25cm] (8) at (16, 2) [thick] {};

\node[draw,circle,inner sep=0.25cm] (20) at (4, 0) [thick] {};
\node[draw,circle,inner sep=0.25cm,label=center:{$r_{s}$}] (30) at (6, 0) [thick] {};
\node[draw,circle,inner sep=0.25cm,label=center:{$r_2$}] (40) at (8, 0) [thick] {};
\node[draw,circle,inner sep=0.25cm,label=center:{$r_1$}] (50) at (10, 0) [thick] {};
\node[draw, fill=yellow,circle,inner sep=0.25cm] (60) at (12, 0) [thick] {};
\node[draw,circle,inner sep=0.25cm,label=center:{$p_3$}] (70) at (14, 0) [thick] {};
\node[draw,circle,inner sep=0.25cm] (80) at (16, 0) [thick] {};

\path (2) edge[ultra thin,-] (3);
\path (3) edge[ultra thin,-]  node[ fill=white, anchor=center, pos=0.5] {$\cdots$} (4);
\path (4) edge[ultra thin,-] (5);
\path (5) edge[ultra thin,-] (6);
\path (6) edge[ultra thin,-] (7);
\path (7) edge[ultra thin, -] node[ fill=white, anchor=center, pos=0.5] {$\cdots$} (8);

\path (20) edge[ultra thin,-] (30);
\path (30) edge[ultra thin,-]  node[ fill=white, anchor=center, pos=0.5] {$\cdots$}(40);
\path (40) edge[ultra thin,-] (50);
\path (50) edge[ultra thin,-] (60);
\path (60) edge[ultra thin,-] (70);
\path (70) edge[ultra thin,-] node[ fill=white, anchor=center, pos=0.5] {$\cdots$} (80);

\path (20) edge[ultra thin,-] (2);
\path (30) edge[ultra thin,-] (3);
\path (40) edge[ultra thin,-] (4);
\path (50) edge[ultra thin,-] (5);
\path (60) edge[ultra thin,-] (6);
\path (70) edge[ultra thin,-] (7);
\path (80) edge[ultra thin,-] (8);
\draw[-] (2) to[out=90,in=90,  looseness=0.1] (8);
\draw[-] (20) to[out=-90,in=-90,  looseness=0.1] (80);
\end{tikzpicture}}
\caption{Only the yellow vertices are labeled.}\label{fig;101ab}
\end{figure}

Using \eqref{reca1} and \eqref{recb1} with induction on $n+k$ and on $n+s+t$, respectively, it is possible to prove the following formulas. We omit the details and leave it as a not-so-easy exercise for the interested reader. Let  $n\geq 2$. For $1\leq k\leq n$ we have
\begin{align*}
a_{n,k}=\begin{cases}
\frac{(n+2)(2n-2)!}{2(n-2)!},& \textnormal{if } k=1,\\
\frac{\binom{n-k+2}{2}(2n-k)!}{2(n-k+1)!},& \textnormal{if } 2\leq k\leq n-1,\\
n!,&\textnormal{if }  k=n,
\end{cases}\end{align*}
For $0\leq s\leq n-1$ we have
\begin{align*}
b_{n,0,s} = b_{n,s,0}=\begin{cases}
\frac{(2n-2)!}{(n-2)!},&\textnormal{if }s=0,\\ \frac{(2n-2-s)!(n-s)}{2(n-1-s)!},&\textnormal{if } 1\leq s\leq n-2,\\(n-1)!,&\textnormal{if } s=n-1,
\end{cases}
\end{align*}
For $1\leq t\leq \left\lfloor(n-1)/2\right\rfloor$ and $t\leq s\leq n-1-t$,
$$b_{n,t,s}=b_{n,s,t}=\frac{(2n-2-s-t)!((n-s-t)(n-s-t+1)+2)}{4(n-s-t)!}.$$
Multiplying $a_{n,1}$ by $2n$ proves the assertion.
\end{proof}

\subsection{Three path graphs joined to form two cycles}
Let $a_1,a_2,a_3\geq 2$ be three integers. Consider the graph, that we denote by $S_{a_1,a_2,a_3}$, that is obtained by connecting three path graphs $P_{a_1}, P_{a_2}$ and $P_{a_3}$, such that two cycles are formed (see figure \ref{fig;klk} below for a visualization).
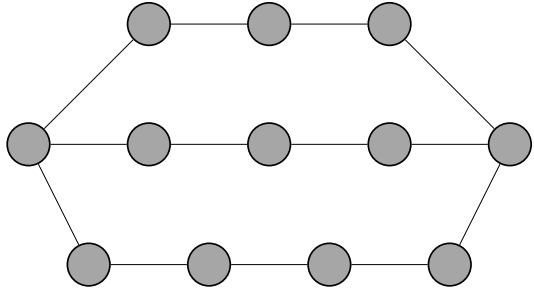
\begin{figure}[H]
\centering
\scalebox{0.8}{
\begin{tikzpicture}[shape = circle, node distance=4cm and 7cm,  nodes={ circle,fill=black!35}]
\node[draw,circle,inner sep=0.25cm] (0) at (-2, 0) [thick] {};
\node[draw,circle,inner sep=0.25cm] (1) at (0, 0) [thick] {};
\node[draw,circle,inner sep=0.25cm] (2) at (2, 0) [thick] {};

\node[draw,circle,inner sep=0.25cm] (3) at (-4, -2) [thick] {};
\node[draw,circle,inner sep=0.25cm] (4) at (-2, -2) [thick] {};
\node[draw,circle,inner sep=0.25cm] (5) at (0, -2) [thick] {};
\node[draw,circle,inner sep=0.25cm] (6) at (2, -2) [thick] {};
\node[draw,circle,inner sep=0.25cm] (7) at (4, -2) [thick] {};

\node[draw,circle,inner sep=0.25cm] (8) at (-3, -4) [thick] {};
\node[draw,circle,inner sep=0.25cm] (9) at (-1, -4) [thick] {};
\node[draw,circle,inner sep=0.25cm] (10) at (1, -4) [thick] {};
\node[draw,circle,inner sep=0.25cm] (11) at (3, -4) [thick] {};
\path (0) edge[ultra thin,-] (1);
\path (1) edge[ultra thin,-] (2);
\path (3) edge[ultra thin,-] (4);
\path (4) edge[ultra thin,-] (5);
\path (5) edge[ultra thin,-] (6);
\path (6) edge[ultra thin,-] (7);

\path (8) edge[ultra thin,-] (9);
\path (9) edge[ultra thin,-] (10);
\path (10) edge[ultra thin,-] (11);

\path (0) edge[ultra thin,-] (3);
\path (3) edge[ultra thin,-] (8);
\path (2) edge[ultra thin,-] (7);
\path (7) edge[ultra thin,-] (11);
\end{tikzpicture}}
\caption{The graph $S_{3,5,4}.$}\label{fig;klk}
\end{figure}

The following theorem establishes the generating function of $\mathcal{L}(S_{a_1,a_2,a_3})$. Its proof is based on the three subsequent lemmas.

\begin{theorem}\label{tha1a2a3}
Let $F(x,y,z)=\sum_{a_1,a_2,a_3\geq2}\mathcal{L}(S_{a_1,a_2,a_3})x^{a_1}y^{a_2}z^{a_3}$. Then
\begin{align*}
&F(x,y,z)=\\&    \frac{16x^2y^2z^2f(x,y,z)}{(1-2x)^3(1-2y)^3(1-2z)^3(1-2x-2y)(1-2x-2z)(1-2y-2z)(1-x-y-z)},
\end{align*}
where $f(x,y,z)=
13+360y^4-96y+292y^2-456y^3-112y^5
+8(56y^3-100y^2+64y-15)(x^5+z^5)
-4(480y^5-2056y^4+3336y^3-2690y^2+1103y-185)xz
+4xz(672y^5-3544y^4+6756y^3-6238y^2+2885y-541)(x+z)
-8xz(160y^5-1304y^4+3204y^3-3520y^2+1864y-393)(x^2+z^2)
-8(384y^5-2592y^4+5960y^3-6436y^2+3418y-727)x^2z^2
-8xz(320y^4-1320y^3+1868y^2-1168y+281)(x^3+z^3)
+16x^2z^2(64y^5-752y^4+2360y^3-3140y^2+1955y-475)(x+z)
-16xz(80y^3-176y^2+138y-39)(x^4+z^4)
+64x^2z^2(32y^4-190y^3+345y^2-262y+74)(x^2+z^2)
+32(128y^4-696y^3+1244y^2-944y+267)x^3z^3
+64x^2z^2(16y^3-50y^2+50y-17)(x^3+z^3)
+128x^3z^3(32y^3-99y^2+98y-33)(x+z)
128x^3z^3(8y^2-12y+5)(x+z)^2
+(496y^5-1824y^4+2596y^3-1852y^2+675y-101)(x+z)
-4(200y^5-892y^4+1470y^3-1183y^2+479y-79)(x^2+z^2)
+4(112y^5-760y^4+1584y^3-1490y^2+679y-124)(x^3+z^3)
+4(224y^4-760y^3+900y^2-474y+97)(x^4+z^4)$.
\end{theorem}

\begin{proof}
Using the results of Lemmas \ref{lem;331}, \ref{lem;332} and \ref{lem;333}, we have \begin{equation}\label{eq;gaq}
\mathcal{L}(S_{a_1,a_2,a_3})=2A_{a_1,a_2,a_3}+2\sum_{s=2}^{a_2-1}B_{a_1,a_2,a_3}(s)+2\sum_{s=1}^{a_1}C_{a_1,a_2,a_3}(s)+2\sum_{s=1}^{a_3}C_{a_3,a_2,a_1}(s).\end{equation}
The assertion follows now (after many algebraic manipulations that we omit) by multiplying (\ref{eq;gaq}) by $x^{a_1}y^{a_2}z^{a_3}$ and summing over all $a_1,a_2,a_3\geq2$.
\end{proof}

For instance, by Theorem \ref{tha1a2a3}, we have 
\begin{align*}
F(x,y,z)&=208x^2y^2z^2+672x^2y^2z^3+752x^2y^3z^2+672x^3y^2z^2+2048x^2y^2z^4+2336x^2y^3z^3\\
&+2544x^2y^4z^2+2496x^3y^2z^3+2336x^3y^3z^2+2048x^4y^2z^2+5952x^2y^2z^5\\
&+6848x^2y^3z^4+8048x^2y^4z^3+8048x^2y^5z^2+8640x^3y^2z^4+8064x^3y^3z^3\\
&+8048x^3y^4z^2+8640x^4y^2z^3+6848x^4y^3z^2+5952x^5y^2z^2+16640x^2y^2z^6\\
&+19200x^2y^3z^5+24048x^2y^4z^4+26720x^2y^5z^3+24048x^2y^6z^2+28160x^3y^2z^5\\
&+26368x^3y^3z^4+26720x^3y^4z^3+26720x^3y^5z^2+33536x^4y^2z^4+26368x^4y^3z^3\\
&+24048x^4y^4z^2+28160x^5y^2z^3+19200x^5y^3z^2+16640x^6y^2z^2+45056x^2y^2z^7\\
&+51968x^2y^3z^6+68592x^2y^4z^5+84816x^2y^5z^4+84816x^2y^6z^3+68592x^2y^7z^2\\
&+87296x^3y^2z^6+82176x^3y^3z^5+84816x^3y^4z^4+87840x^3y^5z^3+84816x^3y^6z^2\\
&+121088x^4y^2z^5+96512x^4y^3z^4+84816x^4y^4z^3+84816x^4y^5z^2+121088x^5y^2z^4\\
&+82176x^5y^3z^3+68592x^5y^4z^2+87296x^6y^2z^3+51968x^6y^3z^2+45056x^7y^2z^2+\ldots.
\end{align*}

In the following three lemmas, we make use of a somewhat unusual notation, intended to reduce clutter. Let $a$ be an integer. Then we write $a^-$ for $a-1$, $a^{--}$ for $a-2$ and $a^{---}$ for $a-3$. Furthermore, if $a,b$ and $c$ are nonnegative integers, we write $\mathcal{M}_{a,b,c}$ for the multinomial coefficient $\binom{a+b+c}{a,b,c}$ and $\mathcal{M}_{a,b}$ for the binomial coefficient $\binom{a+b}{a}$. We only provide full proof for the first lemma, the other proofs being similar.

\begin{lemma}\label{lem;331}
The number of random walk labelings that begin at the yellow vertex of figure \ref{fig;ooo1} is given by
\begin{align*}
&A_{a_1,a_2,a_3}:=\\
&\sum_{k=0}^{a_{1}^-}\sum_{\ell=0}^{a_{3}^-}\M_{a_{2}^{--},k,\ell}\M_{a_{1}-k,a_{3}-\ell}2^{a_{1}^--k+a_{3}^--\ell}+\sum_{k=0}^{a_{3}^-}\M_{a_{2}^{--},k,a_{1}}2^{a_{3}^--k}+\sum_{k=0}^{a_{1}^-}\M_{a_{2}^{--},k,a_{3}}2^{a_{1}^--k}+\M_{a_{2}^{--},a_{1},a_{3}}+\\
&\sum_{k=0}^{a_{1}^-}\sum_{\ell=0}^{a_{2}^{---}}\M_{a_{3},k,\ell}\M_{a_{1}-k,a_{2}^{--}-\ell}2^{a_{1}^--k+a_{2}^{---}-\ell}+\sum_{\ell=0}^{a_{2}^{---}}\M_{a_{1},a_{3},\ell}2^{a_{2}^{---}-\ell}+\\&\sum_{k=0}^{a_{3}^-}\sum_{\ell=0}^{a_{2}^{---}}\M_{a_{1},k,\ell}\M_{a_{3}-k,a_{2}^{--}-\ell}2^{a_{3}^--k+a_{2}^{---}-\ell}.
\end{align*}
\end{lemma}

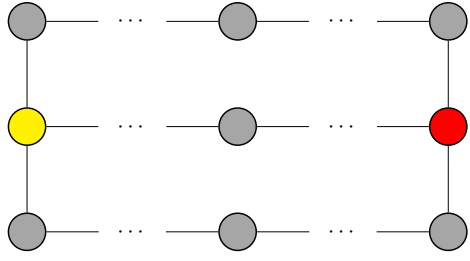
\begin{figure}[H]
\centering
\scalebox{0.7}{
\begin{tikzpicture}[shape = circle, node distance=4cm and 7cm,  nodes={ circle,fill=black!35}]
\node[draw,circle,inner sep=0.25cm] (0) at (-4, 2) [thick] {};
\node[fill=white,inner sep=0.25cm] (1) at (-2, 2) [thick] {$\cdots$};
\node[draw,circle,inner sep=0.25cm] (2) at (0, 2) [thick] {};
\node[fill=white,inner sep=0.25cm] (3) at (2, 2) [thick] {$\cdots$};
\node[draw,circle,inner sep=0.25cm] (4) at (4, 2) [thick] {};

\node[fill=yellow, draw,circle,inner sep=0.25cm] (00) at (-4, 0) [thick] {};
\node[fill=white,inner sep=0.25cm] (10) at (-2, 0) [thick] {$\cdots$};
\node[draw,circle,inner sep=0.25cm] (20) at (0, 0) [thick] {};
\node[fill=white,inner sep=0.25cm] (30) at (2, 0) [thick] {$\cdots$};
\node[fill=red, draw,circle,inner sep=0.25cm] (40) at (4, 0) [thick] {};

\node[draw,circle,inner sep=0.25cm] (000) at (-4, -2) [thick] {};
\node[fill=white,inner sep=0.25cm] (100) at (-2, -2)  {$\cdots$};
\node[draw,circle,inner sep=0.25cm] (200) at (0, -2) [thick] {};
\node[fill=white,inner sep=0.25cm] (300) at (2, -2)  {$\cdots$};
\node[draw,circle,inner sep=0.25cm] (400) at (4, -2) [thick] {};
\path (0) edge[ultra thin,-] (1);
\path (1) edge[ultra thin,-] (2);
\path (2) edge[ultra thin,-] (3);
\path (3) edge[ultra thin,-] (4);

\path (00) edge[ultra thin,-] (10);
\path (10) edge[ultra thin,-] (20);
\path (20) edge[ultra thin,-] (30);
\path (30) edge[ultra thin,-] (40);

\path (000) edge[ultra thin,-] (100);
\path (100) edge[ultra thin,-] (200);
\path (200) edge[ultra thin,-] (300);
\path (300) edge[ultra thin,-] (400);

\path (0) edge[ultra thin,-] (00);
\path (00) edge[ultra thin,-] (000);
\path (4) edge[ultra thin,-] (40);
\path (40) edge[ultra thin,-] (400);
\end{tikzpicture}}
\caption{The setting of Lemma \ref{lem;331}}\label{fig;ooo1}
\end{figure} \noindent

\begin{proof}
We distinguish between two cases.
\begin{enumerate}
    \item [(a)] When the red vertex is labeled, all the vertices of $P_{a_2}$ are labeled: Let $0\leq k\leq a_1-1$ and $0\leq \ell\leq a_3-1$ and assume that when the red vertex is labeled, $k$ vertices of $P_{a_1}$ and $\ell$ vertices of $P_{a_3}$ are labeled. There are $$\binom{a_{2}-2+k+\ell}{a_{2}-2,k,\ell}$$ ways to reach this setting. Now we need to label the remaining $a_1-k$ vertices of $P_{a_1}$ and the remaining $a_3-\ell$ vertices of $P_{a_3}$. Since the red vertex is labeled, as long as there are at least two unlabeled vertices in $P_{a_1}$, there are two ways to pick the next vertex for labeling. This gives $2^{a_1 - k - 1}$ possibilities. Similarly, there are $2^{a_3 - \ell - 1}$ possibilities to label the remaining vertices of $P_{a_3}$. There are $$\binom{a_1-k+a_3-\ell}{a_1-k}$$ ways to alternate between $P_{a_1}$ and $P_{a_3}$. This gives us the first term of $A_{a_1,a_2,a_3}$. The next three terms are obtained by following the same reasoning once with $k=a_1, \ell<a_3$, once with $k<a_1, \ell=a_3$ and once with $k=a_1, \ell=a_3$.
    \item [(b)] When the red vertex is labeled, at least one vertex of $P_{a_2}$ is not labeled: In this case, at least one of $P_{a_1}$ and $P_{a_3}$ must be fully labeled. Assume the latter. Let $0\leq k\leq a_1-1$ and $0\leq \ell\leq a_2-3$ and assume that when the red vertex is labeled, $k$ vertices of $P_{a_1}$ and $\ell$ vertices of $P_{a_2}$ are labeled. Arguing similarly to the previous case, we obtain the fifth term of $A_{a_1,a_2,a_3}$. The assumption that $P_{a_1}$ is fully labeled gives the last term. sixth term is obtained by assuming that both $P_{a_1}$ and $P_{a_3}$ are fully labeled.
\end{enumerate}
\end{proof}

\begin{lemma}\label{lem;332}
Let $2< s< a_2$. The number of random walk labelings that begin at the yellow vertex of figure \ref{fig;ooo2}, such that label of the green vertex is less than the label of the red vertex is given by
\begin{align*}
&B_{a_1,a_2,a_3}(s) :=\\
&\sum_{q=s}^{a_{2}^-}\M_{q^{--},s^{--}}\Bigg[\sum_{k=0}^{a_{1}^-}\sum_{\ell=0}^{a_{3}^-}\M_{a_{2}^--q,k,\ell}\M_{a_{1}-k,a_{3}-\ell}2^{a_{1}^--k+a_{3}^--\ell}+\sum_{k=0}^{a_{1}^-}\M_{a_{2}^--q,k,a_{3}}2^{a_{1}^--k}+\sum_{k=0}^{a_{3}^-}\M_{a_{2}^--q,k,a_{1}}2^{a_{3}^--k}\\
&+\M_{a_{2}^--q,a_{1},a_{3}}\Bigg]+\sum_{q=s}^{a_{2}^{--}}\M_{q^{--},s^{--}}\Bigg[\sum_{\ell=0}^{a_{2}^{--}-q}\sum_{k=0}^{a_{1}^-}\M_{a_{3},k,\ell}\M_{a_{1}-k,a_{2}^--q-\ell}2^{a_{1}^--k+a_{2}^{--}-q-\ell}\\
&+\sum_{\ell=0}^{a_{2}^{--}-q}\sum_{k=0}^{a_{3}^-}\M_{a_{1},k,\ell}\M_{a_{3}-k,a_{2}^--q-\ell}2^{a_{3}^--k+a_{2}^{--}-q-\ell}+\sum_{\ell=0}^{a_{2}^{--}-q}\M_{a_{3},a_{1},\ell}2^{a_{2}^{--}-q-\ell}\Bigg].
\end{align*}
\end{lemma}
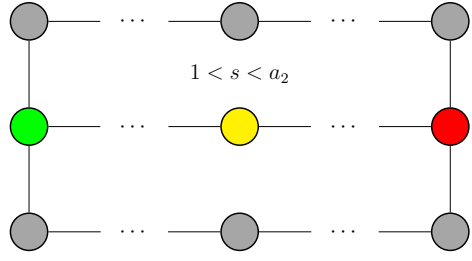
\begin{figure}[H]
\centering
\scalebox{0.7}{
\begin{tikzpicture}[shape = circle, node distance=4cm and 7cm,  nodes={ circle,fill=black!35}]
\node[draw,circle,inner sep=0.25cm] (0) at (-4, 2) [thick] {};
\node[fill=white,inner sep=0.25cm] (1) at (-2, 2) [thick] {$\cdots$};
\node[draw,circle,inner sep=0.25cm] (2) at (0, 2) [thick] {};
\node[fill=white,inner sep=0.25cm] (3) at (2, 2) [thick] {$\cdots$};
\node[draw,circle,inner sep=0.25cm] (4) at (4, 2) [thick] {};

\node[fill=green, draw,circle,inner sep=0.25cm] (00) at (-4, 0) [thick] {};
\node[fill=white,inner sep=0.25cm] (10) at (-2, 0) [thick] {$\cdots$};
\node[label={[label distance=-0.5cm]above:{$1< s< a_2$}},fill=yellow, draw,circle,inner sep=0.25cm] (20) at (0, 0) [thick] {};
\node[fill=white,inner sep=0.25cm] (30) at (2, 0) [thick] {$\cdots$};
\node[fill=red, draw,circle,inner sep=0.25cm] (40) at (4, 0) [thick] {};

\node[draw,circle,inner sep=0.25cm] (000) at (-4, -2) [thick] {};
\node[fill=white,inner sep=0.25cm] (100) at (-2, -2)  {$\cdots$};
\node[draw,circle,inner sep=0.25cm] (200) at (0, -2) [thick] {};
\node[fill=white,inner sep=0.25cm] (300) at (2, -2)  {$\cdots$};
\node[draw,circle,inner sep=0.25cm] (400) at (4, -2) [thick] {};
\path (0) edge[ultra thin,-] (1);
\path (1) edge[ultra thin,-] (2);
\path (2) edge[ultra thin,-] (3);
\path (3) edge[ultra thin,-] (4);

\path (00) edge[ultra thin,-] (10);
\path (10) edge[ultra thin,-] (20);
\path (20) edge[ultra thin,-] (30);
\path (30) edge[ultra thin,-] (40);

\path (000) edge[ultra thin,-] (100);
\path (100) edge[ultra thin,-] (200);
\path (200) edge[ultra thin,-] (300);
\path (300) edge[ultra thin,-] (400);

\path (0) edge[ultra thin,-] (00);
\path (00) edge[ultra thin,-] (000);
\path (4) edge[ultra thin,-] (40);
\path (40) edge[ultra thin,-] (400);
\end{tikzpicture}}
\caption{The setting of Lemma \ref{lem;332}}\label{fig;ooo2}
\end{figure} \noindent

\begin{lemma}\label{lem;333}
Let $1\leq s\leq a_1$. The number of random walk labelings that begin at the yellow vertex of figure \ref{fig;ooo3}, such that label of the green vertex is less than the label of the red vertex is given by
\begin{align*}
&C_{a_1,a_2,a_3}(s) :=\\
&\sum_{q=s}^{a_{1}}\M_{q^-,s^-}\Bigg[\sum_{k=0}^{a_{1}^--q}\sum_{\ell=0}^{a_{3}^-}\M_{a_{2}^{--},k,\ell}\M_{a_{1}-q-k,a_{3}-\ell}2^{a_{1}^--q-k+a_{3}^--\ell}+\sum_{\ell=0}^{a_{3}^-}\M_{a_{2}^{--},a_{1}-q,\ell}2^{a_{3}^--\ell}\\
&+\sum_{k=0}^{a_{1}^--q}\M_{a_{2}^{--},k,a_{3}}2^{a_{1}^--q-k}+\M_{a_{1}-q,a_{2}^{--},a_{3}}+\sum_{\ell=0}^{a_{2}^{---}}\sum_{k=0}^{a_{3}^-}\M_{a_{1}-q,k,\ell}\M_{a_{2}^{--}-\ell,a_{3}-k}2^{a_{3}^--k+a_{2}^{---}-\ell}\\
&+\sum_{\ell=0}^{a_{2}^{---}}\M_{a_{1}-q,a_{3},\ell}2^{a_{2}^{---}-\ell}+\sum_{\ell=0}^{a_{2}^{---}}\sum_{k=0}^{a_{1}^--q}\M_{k,a_{3},\ell}\M_{a_{2}^{--}-\ell,a_{1}-q-k}2^{a_{1}^--q-k+a_{2}^{---}-\ell}\Bigg].
\end{align*}
\end{lemma}

\begin{figure}[H]
\centering
\scalebox{0.7}{
\begin{tikzpicture}[shape = circle, node distance=4cm and 7cm,  nodes={ circle,fill=black!35}]
\node[draw,circle,inner sep=0.25cm] (0) at (-4, 2) [thick] {};
\node[fill=white,inner sep=0.25cm] (1) at (-2, 2) [thick] {$\cdots$};
\node[label={[label distance=-0.5cm]above:{$1\leq s\leq a_1$}},fill=yellow,draw,circle,inner sep=0.25cm] (2) at (0, 2) [thick] {};
\node[fill=white,inner sep=0.25cm] (3) at (2, 2) [thick] {$\cdots$};
\node[draw,circle,inner sep=0.25cm] (4) at (4, 2) [thick] {};

\node[fill=green,draw,circle,inner sep=0.25cm] (00) at (-4, 0) [thick] {};
\node[fill=white,inner sep=0.25cm] (10) at (-2, 0) [thick] {$\cdots$};
\node[draw,circle,inner sep=0.25cm] (20) at (0, 0) [thick] {};
\node[fill=white,inner sep=0.25cm] (30) at (2, 0) [thick] {$\cdots$};
\node[fill=red,draw,circle,inner sep=0.25cm] (40) at (4, 0) [thick] {};

\node[draw,circle,inner sep=0.25cm] (000) at (-4, -2) [thick] {};
\node[fill=white,inner sep=0.25cm] (100) at (-2, -2)  {$\cdots$};
\node[draw,circle,inner sep=0.25cm] (200) at (0, -2) [thick] {};
\node[fill=white,inner sep=0.25cm] (300) at (2, -2)  {$\cdots$};
\node[draw,circle,inner sep=0.25cm] (400) at (4, -2) [thick] {};
\path (0) edge[ultra thin,-] (1);
\path (1) edge[ultra thin,-] (2);
\path (2) edge[ultra thin,-] (3);
\path (3) edge[ultra thin,-] (4);

\path (00) edge[ultra thin,-] (10);
\path (10) edge[ultra thin,-] (20);
\path (20) edge[ultra thin,-] (30);
\path (30) edge[ultra thin,-] (40);

\path (000) edge[ultra thin,-] (100);
\path (100) edge[ultra thin,-] (200);
\path (200) edge[ultra thin,-] (300);
\path (300) edge[ultra thin,-] (400);

\path (0) edge[ultra thin,-] (00);
\path (00) edge[ultra thin,-] (000);
\path (4) edge[ultra thin,-] (40);
\path (40) edge[ultra thin,-] (400);

\end{tikzpicture}}
\caption{The setting of Lemma \ref{lem;333}}\label{fig;ooo3}
\end{figure}
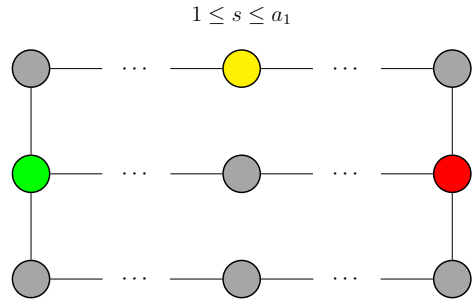 
\end{document}